\documentclass[12pt]{amsart}
\usepackage{tikz-cd}
\usepackage{newfloat}
\usepackage{tikz}
\makeatletter
\DeclareFloatingEnvironment[fileext=dia,within=section]{diagram}

\theoremstyle{plain}

\newtheorem{thm}{Theorem}[section]

\newtheorem{cor}[thm]{Corollary} 

\newtheorem{lem}[thm]{Lemma} 
\newtheorem{prop}[thm]{Proposition} 
\theoremstyle{definition}
\newtheorem{defn}[thm]{Definition}
\theoremstyle{remark}
\newtheorem{rem}[thm]{Remark}
\theoremstyle{remark}

\theoremstyle{remark}

\theoremstyle{remark}

\theoremstyle{definition}

\theoremstyle{definition}

\theoremstyle{plain}

\theoremstyle{definition}

\theoremstyle{remark}

\theoremstyle{remark}

\theoremstyle{definition}

\theoremstyle{remark}
  \newtheorem*{acknowledgement*}{Acknowledgement}

\newcommand{\bigslant}[2]{{\raisebox{.2em}{$#1$}\left/\raisebox{-.2em}{$#2$}\right.}}

\theoremstyle{theorem}
  \newtheorem{thmx}{Theorem}

\newcommand{\R}{{\mathbb R}}

\newcommand{\N}{{\mathbb N}}
\newcommand{\Z}{{\mathbb Z}}

\newcommand{\e}{\varepsilon}

\newcommand{\id}{\mathrm{id}}
\DeclareMathOperator{\tr}{tr}

\newcommand{\ip}[1]{\langle#1\rangle}

\DeclareMathOperator{\Aut}{Aut}

\newcommand{\T}{\mathbb{T}}
\newcommand{\SL}{\mathrm{SL}}
\newcommand{\GL}{\mathrm{GL}}

\DeclareMathOperator{\diag}{diag}

\newcommand{\ca}{\mathcal{A}}

\newcommand{\unit}{\mathbf{1}}
\newcount\colveccount
\newcommand*\colvec[1]{
        \global\colveccount#1
        \begin{pmatrix}
        \colvecnext
}
\def\colvecnext#1{
        #1
        \global\advance\colveccount-1
        \ifnum\colveccount>0
               \\
                \expandafter\colvecnext
        \else
                \end{pmatrix}
        \fi
}

\def\freeprod{\font\bigsymbolsfont=cmsy10 scaled \magstep3
 \setbox0=\hbox{\bigsymbolsfont\char'003 }\mathord{\lower1pt\box0}}\relax\ignorespaces

\newcommand{\Hawaii}{Hawai\kern.05em`\kern.05em\relax i}

\usepackage{amssymb} \usepackage{amscd} \usepackage{hyperref}  \usepackage{tikz} \usepackage{soul}  \usetikzlibrary{matrix,arrows}

\makeatother

\begin{document}

\title{A note on crossed products of rotation algebras} 

\author{Christian B\"onicke}
\address{School of Mathematics and Statistics, University of Glasgow,
\newline University Gardens, Glasgow, G12 8QQ, UK}
\email{christian.bonicke@glasgow.ac.uk}

\author{Sayan Chakraborty}
\address{Department of Mathematics, Indian Institute of Science Education and Research Bhopal, 
\newline Bhopal Bypass Road, Bhopal, 462066, India}
\email{sayan2008@gmail.com}

\author{Zhuofeng He}
\address{Graduate School of Mathematical Sciences, The University of Tokyo,
	\newline 3-8-1 Komaba Meguro-ku Tokyo 153-8914, Japan}
\email{hzf@ms.u-tokyo.ac.jp}

\author{Hung-Chang Liao}
\address{Department of Mathematics and Statistics, University of Ottawa,
	\newline 150 Louis-Pasteur Private, Ottawa, ON K1N 9A7, Canada}
\email{hliao@uottawa.ca}

\date{\today}

\maketitle

\begin{abstract} We compute the $K$-theory of crossed products of rotation algebras $\ca_\theta$, for any real angle $\theta$, by matrices in $\SL(2,\Z)$ with infinite order. Using techniques of continuous fields, we show that the canonical inclusion of $\mathcal{A}_\theta$ into the crossed products is injective at the level of $K_0$-groups. We then give an explicit set of generators for the $K_0$-groups and compute the tracial ranges concretely.
\end{abstract}


\section{introduction}

The rotation algebra $\mathcal{A}_\theta$ associated with a real number $\theta$ is the universal $C^*$-algebra generated by unitaries $U_1$ and $U_2$ satisfying the relation
$$
U_2U_1 = e^{2\pi i\theta}U_1U_2.
$$
Any choice of a matrix
$$
A = \begin{pmatrix}
a & b \\ c & d
\end{pmatrix}
$$
in $\SL_2(\Z)$ determines an automorphism $\alpha_A$ given by
$$
\alpha_A(U_1)=e^{\pi i (ac)\theta} U_1^aU_2^c,\quad\alpha_A(U_2)=e^{\pi i(bd)\theta} U_1^bU_2^d.
$$
This assignment in fact defines an $\SL_2(\Z)$-action on $\ca_\theta$, which was first introduced by Watatani \cite{Wat} and Brenken \cite{Brenken}. Let $\ca_\theta\rtimes_{A}\Z$ denote the crossed product arising from the $C^*$-dynamical system $(\ca_\theta, \Z, \alpha_A)$. In the case $\theta$ is irrational, previous work of the authors \cite{BCHL18} combined an explicit $K$-theory calculation with recent progress in the classification program for simple, separable, nuclear $C^*$-algebras to show that the isomorphism class of $\ca_\theta \rtimes_A \Z$ is completely determined by $\theta$ and the matrix $A$, up to certain (very explicit) equivalence relations.

Crossed products with rational angles present greater challenges. For one thing, the algebras under consideration are not simple, so there is no readily accessible classification theorem. For another, the behavior of the automorphisms $\alpha_A$ and tracial states on the crossed products at the level of $K$-theory are not immediately clear, unlike in the irrational case where all automorphisms on $\ca_\theta$ are easily seen to induce the identity map on $K_0$ and there is only one tracial state on the crossed product.

Despite the difficulties, in \cite{FW94} Farsi and Watling gave a complete $K$-theoretic classification of $\ca_\theta \rtimes_A \Z$ when $\tr(A) = 2$, and proved many interesting partial results in the general case. One crucial ingredient underlying these results is the computation of tracial ranges of the $K_0$-groups, which in turn made use of a short exact sequence relating the tracial ranges and the de la Harpe-Skandalis determinant (see, for example, \cite[Proposition 6]{FW94} and the references therein).

This note represents the authors' attempt to understand this computation at a more concrete level. We identify both the rotational algebras and their crossed products with twisted group $C^*$-algebras, and then apply the technique of homotopic cocycles, as developed by Echterhoff, L\"uck, Phillips, and Walters in \cite{ELPW10}, to show that the canonical inclusion of $\ca_\theta$ inside $\ca_\theta\rtimes_A \Z$ always induces an injective map on $K_0$. Once this is done, an explicit set of generators can be found using the Pimsner-Voiculescu sequence as in the irrational case. It then follows that all tracial states on the crossed products agree on $K_0$, and the tracial range of each generators can be computed directly. The next theorem summarizes our computations (in the case $\tr(A) \neq  2$; the case $\tr(A) = 2$ is similar and treated in Theorem \ref{thm:Ktr=2}).

\begin{thmx}[Theorem \ref{thm:K-theory}]
	Let $\theta$ be any real number in $[0,1]$, and let $A\in \SL_2(\Z)$ be a matrix of infinite order with $\tr(A) \not\in\{0,\pm1,2\}$. Suppose the rank $2$ matrix $I_2-A^{-1}$  has  Smith normal form $\diag(h_1,h_2)$, then
	\begin{align*}
	K_0(\mathcal{A}_{\theta} \rtimes_A \Z) &\cong \Z\oplus \Z, \\
	K_1(\mathcal{A}_{\theta} \rtimes_A \Z) &\cong \Z \oplus \Z \oplus \Z_{h_1} \oplus \Z_{h_2},
	\end{align*}
	and a set of generators for $K_0(\mathcal{A}_{\theta}\rtimes_A\Z)$ is given by $[\unit]_0$ and $[\iota(p_\theta)]_0$, where $\iota:\ca_\theta\to \ca_\theta\rtimes_A \Z$ is the inclusion map and $p_\theta$ is the Rieffel projection associated to $\theta$.
	
	Moreover, given any tracial state $\tau$ on $\ca_\theta\rtimes_A \Z$ the induced map
	$$
	\tau_*: K_0(\mathcal{A}_{\theta}\rtimes_A\Z)\rightarrow \R
	$$ satisfies $\tau_*( [\unit]_0  ) = 1$ and $\tau_*([\iota(p_\theta)]_0) = \theta$. In particular, all tracial states on $\ca_\theta\rtimes_A\Z$ induce the same map on $K_0(\mathcal{A}_{\theta}\rtimes_A\Z)$. 
\end{thmx}

Many results in this note are probably known to experts. Nevertheless it is our hope that our approach and explicit computations make at least some of them more transparent. 

\ \newline
{\bf Acknowledgments}. S. Chakraborty thanks Debashish Goswami for financial support (J C Bose fellowship).


\section{Twisted group $C^*$-algebras}


For our purposes it will be convenient to view both $\ca_\theta$ and $\ca_\theta \rtimes_A \Z$ as twisted group $C^*$-algebras. Here we briefly recall these identifications, and refer the reader to \cite{PR89} and \cite{ELPW10} for more details and other general facts about twisted group $C^*$-algebras. Let $\omega_\theta:\Z^2\times \Z^2 \to \T$ be the $2$-cocycle defined by
$$
\omega_\theta\left( x,y \right) = e^{-\pi i \ip{  \theta x, y }},
$$
where we identify $\theta$ with the $2\times 2$ skew-symmetric matrix $\begin{pmatrix}
0 & -\theta \\ \theta & 0
\end{pmatrix}$. More explicitly, given two elements $x = \begin{pmatrix}
x_1 \\ x_2
\end{pmatrix}$ and $y = \begin{pmatrix}
y_1 \\ y_2
\end{pmatrix}$ in $\Z^2$ we have
$$
\omega_\theta(x,y) = e^{ \pi i \theta( x_2 y_1 - x_1y_2  ) }.
$$
The (full) twisted group $C^*$-algebra $C^*(\Z^2, \omega_{\theta})$ is then a $C^*$-completion of the weighted convolution algebra $\ell^1(\Z^2, \omega_{\theta})$. Using the universal property one checks that there is a $*$-isomorphism
$$
\ca_\theta \xrightarrow{ \cong } C^*(\Z^2, \omega_\theta)
$$
sending each generator $U_i$ to the point mass function $\delta_{e_i}$, where $\{e_1, e_2\}$ is the standard basis for $\Z^2$.

Now consider the action of $\Z$ on $\Z^2$ given by multiplication by the matrix $A$, and write $\Z^2\rtimes_A\Z$ for the resulting semidirect product. Define a $2$-cocycle $\tilde{\omega}_\theta$ on $\Z^2\rtimes_A\Z$  by
$$
\tilde{\omega}_\theta( (x, n), (y, m)  ) = \omega_\theta( x, n.y ).
$$
Then by \cite[Lemma 2.1]{ELPW10} there is an action $\alpha:\Z\to \Aut(C^*(\Z^2, \omega_\theta ))$ generated by the automorphism
$$
\hspace{5cm} [\alpha(f)](x) := f(A^{-1}x) \qquad \qquad (f\in \ell^1(\Z^2), x\in \Z^2)
$$
so that there is a $*$-isomorphism
$$C^*(\Z^2\rtimes_A \Z, \tilde{\omega}_\theta  ) \xrightarrow{\cong} C^*(\Z^2, \omega_{\theta})\rtimes_\alpha \Z.
$$
More concretely, the isomorphism, when described at the level of convolution algebras, is given by
$$
\ell^1(\Z^2\rtimes_A \Z, \tilde{\omega})\to \ell^1( \Z, \ell^1(\Z^2, \omega_{\theta})  ), \qquad f\mapsto \Phi(f),
$$
where $[\Phi(f)](n) = f(\cdot, n)$ for all $n\in \Z$. A direct computation shows that under the identification $\ca_\theta \cong C^*(\Z^2, \omega_\theta)$ the crossed product $C^*(\Z^2, \omega_{\theta})\rtimes_\alpha \Z$ is precisely $\ca_\theta\rtimes_A \Z$ (see \cite[page 185]{ELPW10}). Finally, as all of our groups are amenable, we are free to switch between the full and reduced versions of the crossed products and twisted group $C^*$-algebras \cite[Theorem 3.11]{PR89}.

In what follows we recall a sufficient condition at the level of $2$-cocycles that ensures the resulting twisted group $C^*$-algebras are isomorphic.

\begin{defn}
	Let $G$ be a locally compact group, and let $\omega, \omega'$ be continuous 2-cocycles on $G$. We say $\omega$ and $\omega'$ are \emph{cohomologous}, written as $\omega\sim_{\mathrm{cohom}} \omega'$ if there exists a continuous map $\lambda:G\to \T$ such that
	$$
	\omega(s,t) = \lambda(s)\lambda(t)\overline{\lambda(st)}\omega'(s,t).
	$$
	for all $s,t\in G$.
\end{defn}

Given a continuous $2$-cocycle $\omega$ on a locally compact group $G$ and an automorphism $\Phi\in\Aut(G)$, we define $\omega\circ \Phi: G\times G\to \T$ by
$$
\omega\circ \Phi(s,t) := \omega(\Phi(s),\Phi(t)).
$$
A computation shows that $\omega\circ \Phi$ is also a continuous $2$-cocycle on $G$.

\begin{prop} \label{prop:cohom}
	Let $\omega$ and $\omega'$ be $2$-cocycles on $G$. If there is an automorphism $\Phi\in \Aut(G)$ such that $(\omega \circ \Phi) \sim_{\mathrm{cohom}}\omega'$, then $C^*(G, \omega) \cong C^*(G,\omega')$ and $C^*_r(G, \omega)\cong C^*_r(G, \omega')$.
\end{prop}
\begin{proof}
It is a well-known fact that cohomologous cocycles yield isomorphic twisted group $C^*$-algebras. To complete the proof, one checks that the map $C^*_r(G,\omega)\rightarrow C^*_r(G,\omega\circ \Phi)$ determined by the assignment $f\mapsto f\circ \Phi$ is a $\ast$-isomorphism.
\end{proof}


\section{$K$-theory}

In this section we compute the $K$-theory of the crossed product $\ca_\theta\rtimes_A \Z$, where $A\in \SL_2(\Z)$ is a matrix of infinite order (i.e., $A^n \neq I_2$ for any $n\in \N$). In the case that $\theta$ is irrational, the $K$-theory has been computed in \cite{BCHL18} using the Pimsner-Voiculescu sequence. One crucial ingredient used there is that the automorphism $\alpha_A$ on $\ca_\theta$ induces the identity map at the level of $K_0$ (in fact, this is true for any automorphism of $\ca_\theta$). When $\theta$ is irrational, this is easily seen by examining the image of $K_0$ under the unique tracial state. The same result for rational $\theta$ seems to be known to experts, though we were not able to locate a reference. Here we provide an argument based on homotopies between cocycles, as developed in \cite{ELPW10}.

Let $G$ be a discrete group and $\Omega$ be a $C([0,1],\T)$-valued $2$-cocycle. One can define the twisted crossed product $C^*$-algebra $C([0,1])\rtimes_{ \Omega, r } G$ as in the case of twisted group $C^*$-algebras (see \cite{PR89}). Here the underlying convolution algebra is the algebra of $\ell^1$ functions on $G$ with values in $C([0,1])$. Given any $x\in [0,1]$, the function 
$$
\omega_x := \Omega(\cdot, \cdot)(x)
$$
is a $\T$-valued cocycle on $G$. There is a canonical quotient map (called the \emph{evaluation map})
$$
\mathrm{ev}_x: C([0,1])\rtimes_{ \Omega, r} G\to C^*_r(G, \omega_x)
$$
such that for each function $f\in \ell^1(G, C([0,1]))$ and each $s\in G$ we have $[\mathrm{ev}_x(f)](s) = [f(s)](x)$.

\begin{thm} \cite[Corollary 1.11]{ELPW10} \label{thm:ev}
	If $G$ satisfies the Baum-Connes conjecture with coefficients, then the evaluation map $ev_x$ induces an isomorphism on $K$-theory.
\end{thm}

In principle, this theorem allows us to transfer $K$-theoretical properties from a single fiber to \emph{all} the other fibers. In our case, we make use of our knowledge for irrational $\theta$ and conclude similar properties for the rational ones. As in \cite{BCHL18} we first treat the case when $\tr(A) \neq 2$.

\begin{prop} \label{prop:injectivity}
	Let $\theta$ be any real number in $[0,1]$ and $A\in \SL_2(\Z)$ be a matrix of infinite order with $\tr(A)\notin \{0, \pm 1, 2\}$. Then the inclusion $\ca_\theta\to \ca_\theta\rtimes_A \Z$ induces an isomorphism between $K_0(\ca_\theta)$ and $K_0(\ca_\theta\rtimes_A \Z)$.
	
	As a consequence, the automorphism $\alpha_A$ induces the identity map on $K_0(\ca_\theta)$.
\end{prop}
\begin{proof}
	Let $\Omega$ be the $C([0,1],\T)$-valued cocycle on $\Z^2$ defined by
	$$
	\Omega( \cdot, \cdot   )(\theta) := \omega_\theta
	$$
	and let $\tilde{\Omega}$ be the $C([0,1],\T)$-valued cocycle on $\Z^2\rtimes_A \Z$ defined using $\tilde{\omega}_\theta$ in the analogous way. Since the groups $\Z^2$ and $\Z^2\rtimes_A \Z$ satisfy the Baum-Connes conjecture with coefficients \cite{Hig01}, by Theorem \ref{thm:ev} both evaluation maps
	$$
	\mathrm{ev}_\theta: C([0,1])\rtimes_{ \Omega, r} \Z^2\to C^*_r(\Z^2, \omega_\theta),
	$$
	$$
	\mathrm{ev}_\theta: C([0,1])\rtimes_{ \tilde{\Omega}, r} (\Z^2 \rtimes_A \Z) \to C^*_r(\Z^2 \rtimes_A \Z, \tilde{\omega}_\theta)
	$$
	induce isomorphisms at the level of $K_0$-groups. As in the case of twisted group $C^*$-algebras, there is an identification
	$$
	C([0,1])\rtimes_{ \tilde{\Omega}, r} (\Z^2\rtimes_A \Z) \xrightarrow{\cong} (  C([0,1])\rtimes_{ \Omega, r} \Z^2    )\rtimes_\alpha \Z
	$$
	that respects the evaluation maps (see \cite[Remark 2.3]{ELPW10}). Therefore we obtain a commutative diagram
	\begin{center}
	\begin{diagram}[!ht]
	\begin{tikzpicture}[node distance=2cm, auto]
		\node (1) {$C([0,1])\rtimes_{ \Omega, r} \Z^2$};
		\node (2) [right of=1, node distance=5cm] {$C([0,1])\rtimes_{ \tilde{\Omega}, r} (\Z^2\rtimes_A \Z)$};
		\node (3) [below of=1] {$C^*(\Z^2, \omega_{\theta})$};
		\node (4) [below of=2] {$C^*(\Z^2\rtimes_A \Z, \tilde{\omega}_\theta) $};
		\node (5) [below of=3] {$\ca_\theta$};
		\node (6) [below of=4] {$\ca_\theta\rtimes_A \Z,$};
		\draw[->] (1) to node {$\tilde{\iota}$} (2);
		\draw[->] [swap] (1) to node {$\mathrm{ev}_\theta$} (3);
		\draw[->] (2) to node {$\mathrm{ev}_\theta$} (4);
		\draw[->] (3) to node {$\iota'$} (4);
		\draw[->] (5) to node {$\iota$} (6);
		\draw[->] (3) to node {$\cong$} (5);
		\draw[->] (4) to node {$\cong$} (6);
		\end{tikzpicture},
		\caption{}
		\label{dia:commutative}
\end{diagram}
		\end{center}
	where the horizontal maps are the canonical inclusions. Let us quickly check that the above diagram is indeed commutative. Let $f$ be a function in $\ell^1(\Z^2, C([0,1]) \subseteq C([0,1])\rtimes_{ \Omega, r} \Z^2$. The image $\tilde{\iota}(f)$ is determined by the formula
	$$
	\tilde{\iota}(f)( (s,t),x)=f(s,x)
	$$
	for $s\in \Z^2, t\in \Z$, and $x\in [0,1]$. By the definition of the evaluation maps, we have  
	$$
	[\mathrm{ev}_\theta(f)](s) = f(s,\theta)
	$$
	and 
	$$
	[\mathrm{ev}_\theta({\tilde{\iota}}(f)](s,t)=\tilde \iota(f)((s,t),\theta) = f(s, \theta).
	$$
	By definition $\iota'(h)(s,t) = h(s)$ for any $h \in \ell^1(\Z^2, \omega_{\theta})\subseteq C^*(\Z^2, \omega_{\theta})$, so we have $[\iota'(\mathrm{ev}_{\theta}(f))](s,t) = [\mathrm{ev}_{\theta}(f)](s) = f(s,\theta)$ and the upper rectangle is commutative. The commutativity of the lower rectangle follows from the explicit identifications discussed in Section 2.
	
	  At the level of $K_0$-groups we have
		\begin{center}
		\begin{tikzpicture}[node distance=2cm, auto]
		\node (1) {$K_0(C([0,1])\rtimes_{ \Omega, r} \Z^2)$};
		\node (2) [right of=1, node distance=6cm] {$K_0(C([0,1])\rtimes_{ \tilde{\Omega}, r} (\Z^2\rtimes_A \Z))$};
		\node (3) [below of=1] {$K_0(\ca_\theta  )$};
		\node (4) [below of=2] {$K_0(\ca_\theta \rtimes_A \Z)$};
		\draw[->] (1) to node {$\tilde{\iota}_*$} (2);
		\draw[->] [swap] (1) to node {$(\mathrm{ev}_\theta)_*$} (3);
		\draw[->] (2) to node {$(\mathrm{ev}_\theta)_*$} (4);
		\draw[->] (3) to node {$\iota_*$} (4);
		\end{tikzpicture}.
	\end{center}
	From \cite[Proposition~3.2]{BCHL18} we know that the map $\iota_*$ is an isomorphism when $\theta$ is irrational. As both vertical maps are isomorphisms, so is the map $\tilde{\iota}_*$ at the top. Now the same commutative diagram shows that $\iota_*$ is an isomorphism for any real number $\theta$.
	
	The second statement follows from the Pimsner-Voiculescu exact sequence (see, for example, \cite[Chapter V]{Bla98}). 
\end{proof}

To compute the tracial ranges, let us briefly recall some facts about (the $K$-theory of) $\mathcal{A}_\theta$. For any $\theta$ one can realize $\mathcal{A}_\theta$ as a crossed product $C(\T)\rtimes_{r_\theta} \Z$, where $\Z$ acts on the circle $\T$ by rotations. There is a canonical tracial state $\tau_\theta$ on $\mathcal{A}_\theta$ given by integration. More precisely, if we write $E: C(\T)\rtimes_{r_\theta} \Z\to C(\T)$ for the canonical conditional expectation, the $\tau_\theta$ is given by
$$
\tau_\theta( a  ) := \int_0^{2\pi} [E(a)](x) \;dx.
$$
For any $\theta$ there is a \emph{Rieffel projection} $[p_\theta]_0$ in the $K_0$-group satisfying $(\tau_\theta)_*([p_\theta]_0) = \theta$. If $\theta \neq 0$ then the projection $p_\theta$ was constructed by Rieffel \cite{Rie81}. In the case $\theta = 0$ one has to take $p_\theta$ in the amplification $M_2(\mathcal{A}_\theta)$, and we refer the reader to \cite{Thesis} for details. In either case, the group $K_0(\mathcal{A}_\theta)$ is generated by the elements $[\unit]_0$ and $[p_\theta]_0$.

\begin{rem}
	In Rieffel's original paper the angle $\theta$ was assumed to be irrational, but the construction and the proof in fact works for any nonzero $\theta$.
\end{rem}

We need one more fact about how the tracial states on $\mathcal{A}_\theta$ behave at the $K_0$-level. The following lemma is a special case of \cite[Lemma 2.3]{Ell80}.

\begin{lem} [Elliott]
	Let $\theta$ be any real number in $[0,1)$. Then all tracial states on $\mathcal{A}_\theta$ induce the same map on $K_0(\mathcal{A}_\theta)$.
\end{lem}

\begin{rem}
	In \cite{Ell80} the result was stated in terms of a torsion-free abelian group $G$ and a character $\rho$ on the exterior power $G\wedge G$. In our case one takes the group $G$ to be $\Z^2$, so the exterior power $\Z^2\wedge \Z^2$ is isomorphic to $\Z$. The character $\rho$ is then identified with an element $\e^{2\pi i \theta}$ in the circle. Finally, one checks that the $C^*$-algebra $\mathcal{A}_\rho$ appearing in the lemma is nothing but the rotation algebra $\mathcal{A}_\theta$.
\end{rem}

The next theorem shows that, among other things, the crossed product $\ca_\theta\rtimes_A \Z$ has the same $K$-theory and tracial ranges for any real number $\theta$ in $[0,1]$. We refer the reader to \cite[Section 2.1]{BCHL18} for a discussion of matrix equivalence and Smith normal forms.

\begin{thm}\label{thm:K-theory}
	Let $\theta$ be any real number in $[0,1]$, and let $A\in \SL_2(\Z)$ be a matrix of infinite order with $\tr(A) \not\in\{0,\pm1,2\}$. Suppose the rank $2$ matrix $I_2-A^{-1}$  has the Smith normal form $\diag(h_1,h_2)$, then
	\begin{align*}
	K_0(\mathcal{A}_{\theta} \rtimes_A \Z) &\cong \Z\oplus \Z, \\
	K_1(\mathcal{A}_{\theta} \rtimes_A \Z) &\cong \Z \oplus \Z \oplus \Z_{h_1} \oplus \Z_{h_2},
	\end{align*}
	and a set of generators for $K_0$ is given by $[\unit]_0$ and $[\iota(p_\theta)]_0$. 
	
	Moreover, given any tracial state $\tau'$ on $\ca_\theta\rtimes_A \Z$ the induced map
	$$
	\tau'_*: K_0(\mathcal{A}_{\theta}\rtimes_A\Z)\rightarrow \R
	$$ satisfies $\tau'_*( [\unit]_0  ) = 1$ and $\tau'_*([\iota(p_\theta)]_0) = \theta$. In particular, all tracial states on $\ca_\theta\rtimes_A\Z$ induce the same map on $K_0(\mathcal{A}_{\theta}\rtimes_A\Z)$. 
\end{thm}

\begin{proof}

From Proposition \ref{prop:injectivity} we know that the map  $$\iota_*: K_0(\ca_\theta) \rightarrow K_0(\ca_\theta\rtimes_A \Z)$$ is an isomorphism. Since $[1]_0$ and $[p_\theta]_0$ generate $K_0(\ca_\theta)$, we conclude that $K_0(A_\theta\rtimes_A \Z) \cong \Z\oplus \Z$ and that  $K_0(\ca_\theta\rtimes_A \Z)$ is generated by $[1]_0$ and $\iota_*([p_\theta]_0)$. Now the Pimsner--Voiculescu sequence for $\mathcal{A}_{\theta}\rtimes_A\Z$ reads as follows:
\[
\begin{CD}
K_0 (\mathcal{A}_{\theta}) @>{\id - \alpha^{-1}_{*0}}>>K_0 (\mathcal{A}_{\theta})@>{\iota_{*}}>> K_0 (\mathcal{A}_{\theta} \rtimes_A \Z )   \\
@A{\delta_1}AA & &  @VV{\delta_0}V            \\
 K_1 (\mathcal{A}_{\theta}\rtimes_A \Z) @<<{\iota_{*}}<  K_1 (\mathcal{A}_{\theta}) @<<{\id - \alpha^{-1}_{*1}}< K_1 (\mathcal{A}_{\theta}).
\end{CD}
\]
Since $\iota_*: K_0(\ca_\theta) \rightarrow K_0(\ca_\theta\rtimes_A \Z)$ is an isomorphism, we are left with the short exact sequence
 
 \begin{equation}
0 \longrightarrow \bigslant{ K_1(\mathcal{A}_{\theta})}{\mathrm{im}(\id-\alpha^{-1}_{*1})}\stackrel{\iota_*}{\longrightarrow} K_1(\mathcal{A}_{\theta}\rtimes_A\Z) \stackrel{\delta_1}{\longrightarrow} K_0(\mathcal{A}_{\theta}) \longrightarrow 0.
\end{equation}
Now the same proof as in \cite[Proposition 3.1]{BCHL18} gives that
$$
K_1(\mathcal{A}_{\theta} \rtimes_A \Z) \cong \Z \oplus \Z \oplus \Z_{h_1} \oplus \Z_{h_2}.
$$ 
 
For the second statement, let $\tau'$ be any tracial state on $\mathcal{A}_{\theta}\rtimes_A\Z$. Then $\tau'\circ \iota$ is a tracial state on $\mathcal{A}_{\theta}$. By the previous theorem any tracial state on $\mathcal{A}_{\theta}$ induces the same map from $K_0(\mathcal{A}_{\theta})$ to $\R$, and the map sends $[p_\theta]_0$ to $\theta$. Therefore,
\begin{align*}
\tau'_*(\iota_*([p_\theta]_0))&=(\tau\circ \iota)_*([p_\theta]_0) = \theta.
\end{align*}
This concludes the proof.
\end{proof}

The following result was first obtained by Farsi and Watling in \cite{FW94}.

\begin{cor} \label{cor:same_theta} \cite[Proposition 6]{FW94}
	Let $\theta$ and $\theta'$ be two real numbers, and $A$ and $B$ two matrices in $\SL_2(\Z)$ such that $\tr(A), \tr(B)\not\in\{0,\pm1,2\}$. If $\mathcal{A}_{\theta}\rtimes_A\Z$ is isomorphic to $\mathcal{A}_{\theta'}\rtimes_B\Z$, then 
	$$
	\theta\equiv \pm\theta'\mod \Z\text{\ and\ } I_2-A^{-1}\sim_{\mathrm{eq}}I_2-B^{-1}.
	$$
\end{cor}
\begin{proof}
	From the computation of $K$-theory (Theorem \ref{thm:K-theory}), we obtain the matrix equivalence between $I_2-A^{-1}$ and $I_2-B^{-1}$ directly from the isomorphic $K_1$-groups.
	
	The condition that $\theta = \pm \theta'$ can be viewed as a consequence of isomorphic crossed products having the same \emph{twist}, as defined in \cite{Yin90}. Here we give a direct argument for the reader's convenience. Suppose $\varphi: \mathcal{A}_{\theta}\rtimes_A\Z \rightarrow \mathcal{A}_{\theta'}\rtimes_B\Z$ is an isomorphism. Then the diagram
	\begin{center}
		\begin{tikzpicture}[node distance=2cm, auto]
		\node (1) {$K_0(\mathcal{A}_{\theta}\rtimes_A\Z)$};
		\node (2) [right of=1, node distance=3.5cm] {$K_0(\mathcal{A}_{\theta'}\rtimes_B\Z)$};
		\node (3) [below of=1] {$\R$};
		\node (4) [below of=2] {$\R$};
		\draw[->] (1) to node {$\varphi_*$} (2);
		\draw[->] [swap] (1) to node {$(\tau_A)_*$} (3);
		\draw[->] (2) to node {$(\tau_B)_*$} (4);
		\draw[->] (3) to node {$=$} (4);
		\end{tikzpicture}
	\end{center}
	commutes because of Theorem \ref{thm:K-theory}. The isomorphism $\varphi_*$ is hence represented by a matrix of the form
	$$
	\begin{pmatrix}
	1&n\\
	0&\pm1
	\end{pmatrix}.
	$$
	We compute that
	\begin{align*}
		\theta&= (\tau_A)_*([\iota(p_{\theta})]_0) = (\tau_B)_*\varphi_*([\iota(p_{\theta})]_0)\\
		&= (\tau_B)_*(n[\unit]_0\pm[\iota(p_{\theta'})]_0)\\
		&=n\pm\theta'.
	\end{align*}
	Hence $\theta\equiv\pm\theta'\pmod\Z$ and the proof is complete.
\end{proof}

\begin{rem}
	The twist invariant, as mentioned in the proof of Corollary \ref{cor:same_theta}, is only defined for $C^*$-algebras with certain $K$-theoretic properties. One of the consequences of Theorem \ref{thm:K-theory} is that the twist is well-defined for the class of crossed products that we are considering. 
\end{rem}

We conclude the section with the $K$-theory computation of $\mathcal{A}_{\theta} \rtimes_A \Z$ when  $\tr(A) = 2$ and $A\neq I_2$. The argument is essentially the same to that of the irrational cases investigated in \cite{BCHL18}. To avoid undue repetition, we only give a brief outline.

Any matrix $A$ in $\SL_2(\Z)$ with $\tr(A)=2$ and $A\neq I_2$ is conjugate in $\SL_2(\Z)$ to either 
\begin{equation}\label{reduction}
\begin{pmatrix}
1&h_1\\
0&1
\end{pmatrix}\ \text{or\ } 
\begin{pmatrix}
1&0\\
h_1&1
\end{pmatrix},
\end{equation}
for some positive integer $h_1$ (see \cite{MKS04}). Therefore, the computation reduces to the special cases when $A$ is equal to either matrix in (\ref{reduction}). At the $K_0$-level, Diagram \ref{dia:commutative} yields the commutative diagram
\begin{center}
		\begin{tikzpicture}[node distance=2cm, auto]
		\node (1) {$K_0(C([0,1])\rtimes_{ \Omega, r} \Z^2)$};
		\node (2) [right of=1, node distance=6cm] {$K_0(C([0,1])\rtimes_{ \tilde{\Omega}, r} (\Z^2\rtimes_A \Z))$};
		\node (3) [below of=1] {$K_0(\ca_\theta  )$};
		\node (4) [below of=2] {$K_0(\ca_\theta \rtimes_A \Z).$};
		\draw[->] (1) to node {$\tilde{\iota}_*$} (2);
		\draw[->] [swap] (1) to node {$(\mathrm{ev}_\theta)_*$} (3);
		\draw[->] (2) to node {$(\mathrm{ev}_\theta)_*$} (4);
		\draw[->] (3) to node {$\iota_*$} (4);
		\end{tikzpicture}.
	\end{center}
	 Since the homomorphism $\iota_*$ is known to be injective when $\theta$ is irrational \cite{BCHL18}, the same must be true for $\tilde{\iota}_*$. Applying commutativity of the diagram again we see that $\iota_*$ is injective for any $\theta$. Then the Pimsner--Voiculescu sequence gives the (split) short exact sequence
\begin{equation*}\label{tr_2}
0 \longrightarrow  K_0(\mathcal{A}_{\theta}) \stackrel {\iota_*}{\longrightarrow} K_0(\mathcal{A}_{\theta}\rtimes_A\Z)\stackrel{\delta_0}{\longrightarrow} \mathrm{ker}(\id-\alpha^{-1}_{*1}) \longrightarrow 0.
\end{equation*}
The assumption on $A$ implies that $\mathrm{ker}(\id-\alpha^{-1}_{*1})=\Z[U_i]_1$ for $i=1,2$. Now it is a complete repetition of the proof of \cite[Proposition 3.3]{BCHL18} and the discussion that precedes it to find a preimage of $[U_i]_1$ under $\delta_0$. Essentially one computes the preimage $[P]_0$ of $[U_i]_1$ for the case $\theta = 0$, which was done in \cite{Thesis}. Then choosing a suitable $*$-homomorphism from $C(\T)$ to $\ca_\theta$ and applying naturality of the Pimsner-Voiculescu sequence allows one to find the preimage $[P_{U_1,w}]_0$ of $[U_i]_1$ for any $\theta\in [0,1]$. Finally, from the proof of \cite[Proposition 3.3]{BCHL18} we see that the trace of $[P_{U_1,w}]_0$ does not depend on the angle $\theta$.

The next theorem summarizes our discussion for the case $\tr(A) = 2$. We write $P_A$ for the image of $P_{U_i,w}$ under the isomorphism coming from conjugacy of matrices.

\begin{thm}\label{thm:Ktr=2}
	Let $\theta$ be any real number in $[0,1]$, and let $A\in \SL_2(\Z)$ be a matrix of infinite order with $\tr(A) = 2$. Suppose the rank-one matrix $I_2-A^{-1}$ has the Smith normal form  $\diag(h_1,0)$. Then
		\begin{align*}
			K_0(\mathcal{A}_{\theta} \rtimes_A \Z) &\cong \Z\oplus \Z\oplus \Z, \\
			K_1(\mathcal{A}_{\theta} \rtimes_A \Z) &\cong \Z\oplus \Z \oplus \Z \oplus \Z_{h_1}
		\end{align*}
	and a set of generators for $K_0$ is given by $[\unit]_0$, $\iota_*([p_\theta]_0)$, and $[P_A]_0$. 
	
	Moreover, given any tracial state $\tau'$ on $\ca_\theta\rtimes_A \Z$ the induced map
$$
\tau'_*: K_0(\mathcal{A}_{\theta}\rtimes_A\Z)\rightarrow \R
$$ satisfies $\tau'_*( [\unit]_0  ) = 1$, $\tau'_*([\iota(p_\theta)]_0) = \theta$, and $\tau'_*( [P_A]_0  ) = 1$. In particular, all tracial states on $\ca_\theta\rtimes_A\Z$ induce the same map on $K_0(\mathcal{A}_{\theta}\rtimes_A\Z)$. 
\end{thm}

\begin{rem}
 The main issue that arises when one attempts to completely classify the crossed products considered in this article, is the problem of producing $\ast$-isomorphisms 
$$\mathcal{A}_\theta\rtimes_A \Z \rightarrow \mathcal{A}_{\theta'}\rtimes_B \Z,$$
provided that $\theta,\theta'$ and $A$ and $B$ are suitably related.
When the angles are assumed to be irrational, such isomorphisms were produced in \cite{BCHL18} by invoking abstract classification machinery, which reduces the problem to producing isomorphisms between the respective Elliott-invariants.

In the rational case this approach is no longer available. Instead, one can try to produce such isomorphisms directly on the level of the crossed products. This however turns out to be very difficult. One possible approach is to further exploit the realization of the algebras $\mathcal{A}_\theta\rtimes_A \Z$ as twisted group algebras of the form $C^*(\Z^2\rtimes_A \Z,\tilde{\omega}_\theta)$. We will discuss this approach in two directions.

Suppose first that $A\in \SL_2(\Z)$ admits a reversing symmetry $B\in \GL_2(\Z)$ (i.e. $BA = A^{-1}B$) such that $\det(B)=-1$. Routine calculations reveal that $B$ defines an automorphism $\Phi_B\in Aut(\Z^2\rtimes_A \Z)$ such that
$
\tilde{\omega}_\theta = \tilde{\omega}_{-\theta}\circ \Phi_B.$ In particular, under these assumptions, we can conclude that the crossed products $\mathcal{A}_\theta \rtimes_A \Z$ and $\mathcal{A}_{\theta'}\rtimes_A \Z$ are isomorphic if and only if $\theta = \pm \theta' \pmod \Z$.
It is known, that not all hyperbolic matrices admit reversing symmetries. For example, it was shown in \cite{BR97} that the hyperbolic matrix $$\begin{pmatrix}
4 & 9 \\
7 & 16
\end{pmatrix}$$ admits no reversing symmetry. However, many hyperbolic matrices $A =\begin{pmatrix}
a & b \\ c & d
\end{pmatrix}\in \SL_2(\Z)$ do admit reversing symmetries with determinant $-1$, in particular those for which either $b | (a-d)$ or $c|(a-d)$ (see \cite[3.1.1, Proposition 4]{BR97}).	
\end{rem}
   
\begin{rem}
	In a different direction, let us explain the role that the structure of the group $\Z^2\rtimes_A \Z$ plays. Our main observation is, that the center of $\Z^2\rtimes_A \Z$ is non-trivial if and only if $\tr(A)=2$. In that case, one easily checks that $Z(\Z^2\rtimes_A \Z)\cong \Z$ and that the quotient of $\Z^2\rtimes_A \Z$ by its center is isomorphic to $\Z^2$.
In particular, $\Z^2\rtimes_A \Z$ is a nilpotent group of nilpotency class $2$.
This observation allows us to apply recent results on $\mathrm{C}^*$-superrigidity obtained in \cite{MR3861703}: indeed, given two matrices $A,B\in \SL_2(\Z)$ with $\tr(A)=2$, a careful comparison of the equivalence classes of the triples $(Z(\Z^2\rtimes_A \Z),\Z^2\rtimes_A \Z/Z(\Z^2\rtimes_A \Z),\sigma_A)$ and $(Z(\Z^2\rtimes_B \Z),\Z^2\rtimes_B \Z/Z(\Z^2\rtimes_B \Z),\sigma_B)$, where $\sigma_A$ and $\sigma_B$ denote the respective extension cocycles, reveals, that $C(\T^2)\rtimes_A \Z\cong C^*(\Z^2\rtimes_A\Z)\cong C^*(\Z^2\rtimes_B\Z)\cong C(\T^2)\rtimes_B \Z$ if and only if $\tr(B)=2$ and $I_2-A^{-1}$ and $I_2-B^{-1}$ share the same Smith normal form up to sign.
This provides an interesting alternative approach to \cite[Theorem~20]{FW94} in the case $\theta=0=\theta'$.  
\end{rem}

\bibliographystyle{plain}
\bibliography{Biblio-Database}

\end{document}